\documentclass[11pt]{article}

\usepackage{amsmath}
\usepackage{amsfonts}
\usepackage{color}
\usepackage{amssymb}
\usepackage{graphicx}
\usepackage{hyperref}
\usepackage{enumerate}
\usepackage[all]{xy}

\def\tr{{\rm tr}}

 \allowdisplaybreaks

\begin{document}

\newtheorem{problem}{Problem}

\newtheorem{theorem}{Theorem}[section]
\newtheorem{corollary}[theorem]{Corollary}
\newtheorem{definition}[theorem]{Definition}
\newtheorem{conjecture}[theorem]{Conjecture}
\newtheorem{question}[theorem]{Question}
\newtheorem{lemma}[theorem]{Lemma}
\newtheorem{proposition}[theorem]{Proposition}
\newtheorem{quest}[theorem]{Question}
\newtheorem{example}[theorem]{Example}

\newenvironment{proof}{\noindent {\bf
Proof.}}{\rule{2mm}{2mm}\par\medskip}

\newenvironment{proofof3}{\noindent {\bf
Proof of  Theorem 1.2.}}{\rule{2mm}{2mm}\par\medskip}

\newenvironment{proofof5}{\noindent {\bf
Proof of  Theorem 1.3.}}{\rule{2mm}{2mm}\par\medskip}

\newcommand{\remark}{\medskip\par\noindent {\bf Remark.~~}}
\newcommand{\pp}{{\it p.}}
\newcommand{\de}{\em}

\title{  {Extensions of Fiedler-Markham's inequality and Thompson's inequality}
\thanks{  $\dag$ Corresponding author. 
 E-mail addresses: ytli0921@hnu.edu.cn (Y. Li), FairyHuang@csu.edu.cn (Y. Huang), 
fenglh@163.com (L. Feng),  wjliu6210@126.com (W. Liu).}  }

\author{Yongtao Li$^a$, Yang Huang$^{b}$, Lihua Feng$^b$,  Weijun Liu$^{\dag,b}$\\
{\small ${}^a$School of Mathematics, Hunan University} \\
{\small Changsha, Hunan, 410082, P.R. China } \\
{\small $^b$School of Mathematics and Statistics, Central South University} \\
{\small New Campus, Changsha, Hunan, 410083, P.R. China. } }

\maketitle

\vspace{-0.5cm}

\begin{abstract}

We  present some new inequalities related to determinant and trace 
for positive semidefinite block matrices by using symmetric tensor product, 
which are extensions of Fiedler-Markham's inequality and  Thompson's inequality.   
 \end{abstract}

{{\bf Key words:}   Positive semidefinite matrices; Fiedler and Markham's inequality; \\ 
Thompson's inequality.  } \\
{2010 Mathematics Subject Classication.  15A45, 15A60, 47B65.}

\section{Introduction}
Throughout the paper, we use the following standard notation. 
The set of $n\times n$ complex matrices is denoted by $\mathbb{M}_n(\mathbb{C})$, 
or simply by $\mathbb{M}_n$,  
and the identity matrix of order $n$ by  $I_n$, or $I$ for short. 
In this paper, 
we are interested in complex block matrices. Let $\mathbb{M}_n(\mathbb{M}_k)$ 
be the set of complex matrices partitioned into $n\times n$ blocks 
with each block being a $k\times k$ matrix. 
The element of $\mathbb{M}_n(\mathbb{M}_k)$ is usually written as ${ H}=[H_{ij}]_{i,j=1}^n$, 
where $H_{ij}\in \mathbb{M}_k$ for all $i,j$. 
By convention, if $X\in \mathbb{M}_n$ is positive semmidefinite, 
we write $X\ge 0$. For two Hermitian matrices $A$ and $B$ of the same size, 
$A\ge B$ means $A-B\ge 0$. 

Let $H=[H_{ij}]_{i,j=1}^n$ be positive semidefinite. 
It is well known that 
both $[\mathrm{det}\, H_{ij} ]_{i,j=1}^n$ and $[\mathrm{tr}\, H_{ij}]_{i,j=1}^n$ 
are positive semidefinite;  see, e.g., \cite{Zha12}. 
Moreover, the renowned Fischer's inequality (see \cite[p. 506]{HJ13} or \cite[p. 217]{Zhang11}) says that 
\begin{equation} \label{eqfis}
 \prod_{i=1}^n \det H_{ii} \ge \det H. 
\end{equation}
There are various extensions and generalizations of (\ref{eqfis}) in the literature, 
e.g., \cite{CTZ19, Eve58, SHGJ06, JM85,LD14,deP71}.
In 1961, Thompson \cite{Tho61} generalized Fischer's determinantal inequality as below (\ref{eq1}) 
by an identity of Grassmann products; see \cite{LS16} for a short proof.

\begin{theorem}\label{thm11} 
Let ${ H}=[H_{ij}]_{i,j=1}^n \in \mathbb{M}_n(\mathbb{M}_k)$ be positive semidefinite. Then \begin{eqnarray}\label{eq1}
	\det ([\det H_{ij}])\geq \det H.
	\end{eqnarray}
\end{theorem}

Indeed, (\ref{eq1}) is a generalization of Fischer's result (\ref{eqfis}) 
since we can get by a special case of Fischer's inequality 
that $\prod_{i=1}^n \det H_{ii}\ge \det ([\det H_{ij}])$. 
In 1994, Fiedler and Markham \cite{FM94} revisited Thompson's result 
and proved the following inequality for trace.

\begin{theorem} \label{thm12}
Let ${ H}=[H_{ij}]_{i,j=1}^n \in \mathbb{M}_n(\mathbb{M}_k)$ be positive semidefinite. Then 
\begin{equation} \label{eqfm}
\left(\frac{\det \bigl( [\mathrm{tr} H_{ij}] \bigr)}{k}\right)^k \ge \det { H}.
\end{equation}
\end{theorem}

In fact, Lin \cite{LinZhang, Lin16} pointed out that in their proof of Theorem \ref{thm12}, 
Fiedler and Markham used the superadditivity of determinant functional, 
which can be improved by Fan-Ky's determinantal 
inequality (see \cite{FanKy} or \cite[p. 488]{HJ13}), 
i.e., the log-concavity of the determinant over the cone of positive semidefinite matrices. 
 Here we state the improved version (\ref{eqlin}) as follows;    
see \cite{Kua17, Li20} for a 
short proof and extension to the class of sector matrices.  

\begin{theorem}\label{Lin}
Let ${ H}=[H_{ij}]_{i,j=1}^n \in \mathbb{M}_n(\mathbb{M}_k)$ be positive semidefinite. Then 
\begin{equation} \label{eqlin}
\left(\frac{\det \bigl( [\mathrm{tr} H_{ij}] \bigr)}{k^n}\right)^k \ge \det { H}.
\end{equation}
\end{theorem}

The paper is organized as follows. 
In Section \ref{sec2}, for convenience, we briefly review some basic definitions and properties of 
symmetric tensor product in Multilinear Algebra Theory. 
In Section \ref{sec3}, we show two extensions of Fiedler-Markham's inequality 
by using symmetric tensor product (Theorem \ref{thm35} and Theorem \ref{thm37}). 
Additionally, some other determinantal inequalities of positive semidefinite block matrices are included. 
In Section \ref{sec4}, 
we  give an extension of Thompson's inequality (Theorem \ref{thm38}), 
which also yields a generalization of Fischer's inequality (Corollary \ref{coro39}).

\section{Preliminaries}
\label{sec2}

Before starting our results, we first review some basic definitions and 
notations of multilinear algebra \cite{Me97}. 
If $A=[a_{ij}]$ is of order $m\times n$ and $B$ is $s\times t$, the tensor product of $A,B$, 
denoted by $A\otimes B$, is an $ms\times nt$ matrix, 
partitioned into $m\times n$ block matrix with the $(i,j)$-block the $s\times t$ matrix $a_{ij}B$. 
Let $\otimes^r A:=A\otimes \cdots \otimes A$ be the $r$-fold tensor power of $A$. 
Let $V$ be an $n$-dimensional Hilbert space and 
$\otimes^rV$ be the tensor product space of $r$ copies of $V$. 
The symmetric tensor product of vectors $v_1,v_2,\ldots ,v_r$ in $V$ is defined as 
\[ v_1\vee v_2 \vee \cdots \vee v_r := \frac{1}{\sqrt{r!}} \sum_{\sigma } 
v_{\sigma (1)}\otimes v_{\sigma (2)} \otimes \cdots \otimes v_{\sigma (r)}, \]
where $\sigma$ runs over all permutations of the $r$ indices. 
The linear span of all these vectors comprises the subspace $\vee^rV$ of $\otimes^r V$, 
this is  called the $r$th symmetric tensor power of $V$. 
Let $A$ be a linear map on $V$, then $(\otimes^r A)(v_1\vee \cdots \vee v_r)= 
Av_1\vee \cdots \vee Av_r$ lies in $\vee^r V$ 
for all $v_1,\ldots ,v_r$ in $V$. Therefore, the subspace $\vee^r V$ 
is invariant under the tensor operator $\otimes^r A$. 
The restriction of $\otimes^r A$ to this invariant subspace is denoted by $\vee^r A$ 
and called the $r$th symmetric tensor power of $A$; 
see \cite[pp. 16-19]{Bha97} and \cite{Me97} for more details. 
We denote by $s_r(A)$ the $r$th complete symmetric polynomial  
of the eigenvalues of  $A\in \mathbb{M}_n(\mathbb{C})$, i.e, 
\[ s_r(A):=\sum_{1\le i_1\le i_2\le \cdots \le i_r\le n} \lambda_{i_1}(A)\lambda_{i_2}(A)\cdots 
\lambda_{i_r}(A). \]

Some basic properties of  tensor product 
are summarised below.

\begin{proposition}\label{prop21}
Let $A, B, C$ be matrices of sizes $n\times n$. Then 
\begin{enumerate}
	\item[(1).]  \quad $\otimes^r(AB)=(\otimes^rA)( \otimes^r B)$ 
and $\vee^r(AB)=(\vee^rA)( \vee^r B)$.
	\item[(2).] \quad  $\tr (\otimes^rA)= (\tr A)^r$ and $\tr (\vee^rA)=s_r(A)$.
	\item[(3).]  \quad $\det (\otimes^r A)=(\det A)^{rn^{r-1}}$ 
and $\det (\vee^rA)=(\det A)^{ \frac{r}{n}{n+r-1 \choose r}}$.
\end{enumerate}
Furthermore, if $A,B,C$ are positive semidefinite matrices, then 
\begin{enumerate}
	\item[(4).] \quad $A\otimes B$ and $A\vee B$ are positive semidefinite. 
\item[(5).] \quad If $A\ge B$, then $A\otimes C\ge B\otimes C$ and 
$A\vee C\ge B\vee C$. 
\item[(6).] \quad $\otimes^r (A+B)\ge \otimes^r A+\otimes^r B$ and 
$\vee^r (A+B)\ge \vee^r A+\vee^r B$ for all positive integer $r$. 
\end{enumerate}
\end{proposition}

In this paper, we are mainly investigate positive semidefinite block matrices. 
For $H=[H_{ij}]\in \mathbb{M}_n(\mathbb{M}_k)$, 
we denote by $T_n^r(H):=[\otimes^r H_{ij}]\in \mathbb{M}_n(\mathbb{M}_{k^r} )$ and 
$Q_n^r(H):=[\vee^r H_{ij}]\in \mathbb{M}_n(\mathbb{M}_{{k+r-1 \choose r} })$.

\section{Extensions of Fiedler-Markham's inequality}
\label{sec3}

In the section, we first prove some lemmas for latter use, 
and then we give two extensions of Fiedler-Markham's inequality. 

\begin{lemma} \label{lem31}
Let $H=[H_{ij}]\in \mathbb{M}_n(\mathbb{M}_k)$. 
Then $T_n^r(H)$ is a principal submatrix of $\otimes^r H$.
\end{lemma}

\begin{proof}
Without loss of generality, we may write $H=X^*Y$, where $X,Y$ are $nk\times nk$. 
Now we partition $X=(X_1,X_2,\ldots ,X_n)$ and $Y=(Y_1,Y_2,\ldots ,Y_n)$ 
with each $X_i,Y_i$ is an $nk\times k$ complex matrix. 
Under this partition, we see that $H_{ij}=X_i^*Y_j$. 
Also we have $Y_j=YE_j$, where $E_j$ is a suitable $nk\times k$ matrix 
such that its $j$-th block  is extractly $I_k$ and otherwise $0$.  
By (1) of Proposition \ref{prop21}, we obtain
\[ \otimes^r H_{ij}=\otimes^r (X_i^*Y_j )= \otimes^r (E_i^*X^*YE_j)
=(\otimes^r E_i)^* (\otimes^r (X^*Y)) (\otimes^r E_j).\]
In other words, 
\[ [\otimes^r H_{ij}]_{i.j=1}^n =E^*(\otimes^rA)E, 
\quad E=[\otimes^r E_1,\otimes^r E_2,\ldots ,\otimes^r E_n]. \]
It is easy to verify that $E$ is a permutation matrix with $1$ only in diagonal entries.
\end{proof}

\begin{corollary} \label{coro32}
If $H \in \mathbb{M}_n(\mathbb{M}_k)$ is positive semidefinite, 
then so are $T_n^r(H)$ and $Q_n^r(H)$. 
\end{corollary}

\begin{proof}
As $H$ is positive semidefinite, 
so are $\otimes^r H$ and $\vee^r H$. 
By Lemma \ref{lem31}, we can see that $T_n^r(H)$ and $Q_n^r(H)$ are positive semidefinite. 
\end{proof}

\begin{lemma} \label{lem33}
Let $A,B\in \mathbb{M}_n(\mathbb{M}_k)$ be positive semidefinite. 
Then for $r\in \mathbb{N}^*$
\[ T_n^r(A+B) \ge T_n^r(A) +T_n^r(B),\]
and 
\[ Q_n^r(A+B) \ge Q_n^r(A) +Q_n^r(B). \]
\end{lemma}

\begin{proof}
By the basic property of tensor power, Proposition \ref{prop21}, we have 
\[ \otimes^r (A+B)\ge \otimes^r A+\otimes^r B. \]
Since $[\otimes^r A_{ij}]_{i,j=1}^n $ 
is a principal submatrix of $\otimes^r A$, Lemma \ref{lem31}, 
it yields 
\begin{equation*}
  [\otimes^r (A_{ij}+B_{ij})]_{i,j=1}^n \ge [\otimes^r A_{ij}]_{i,j=1}^n + 
[\otimes^r B_{ij}]_{i,j=1}^n.  
\end{equation*}
By restricting above inequality to the symmetric tensors, we obtain 
\[  [\vee^r (A_{ij}+B_{ij})]_{i,j=1}^n \ge [\vee^r A_{ij}]_{i,j=1}^n + 
[\vee^r B_{ij}]_{i,j=1}^n.  \]
This completes the proof. 
\end{proof}

The following Proposition \ref{prop34} 
is a key step in proof of our extensions (Theorem \ref{thm35}), 
and it can be regarded  as a Thompson-type determinantal inequality.   

 \begin{proposition}\label{prop34} 
Let $H=[H_{ij}]\in \mathbb{M}_{n}(\mathbb{M}_{k})$ be positive  definite. 
Then for $r\in \mathbb{N}^*$  
\begin{eqnarray}\label{eq4}
 	 \det T_n^r(H) \ge (\det H)^{ rk^{r-1}}.
 	\end{eqnarray}  
\end{proposition}

 \begin{proof}   
Since the determinant functional is continuous, 
we may assume without loss of generality that $H$ is positive definite 
by a standard perturbation argument.  
 As $H$ is positive definite, we may further write $H=T^*T$ with $T=[T_{ij}]\in \mathbb{M}_{n}(\mathbb{M}_{k})$ being block upper triangular matrix, see \cite[p. 441]{HJ13}. 
Note that 
\begin{eqnarray*}
	(\det H)^{rk^{r-1}}
&=& (\det T^{*}T)^{rk^{r-1}}
 =\Bigl( \prod^{n}_{i=1}\det T_{ii}^{*} \cdot\prod^{n}_{i=1}\det T_{ii}    \Bigr)^{rk^{r-1}}   \\
	&=& \prod_{i=1}^{n}(\det T_{ii}^{*})^{rk^{r-1}} 
      \cdot \prod_{i=1}^{n}(\det T_{ii})^{rk^{r-1}} \\
	&=&\prod_{i=1}^{n}\det(\otimes ^{r}T_{ii}^{*})\prod_{i=1}^{n}\det(\otimes^{r}T_{ii}),
\end{eqnarray*}
where the last equality is by Proposition \ref{prop21}. 
We next may assume $T_{ii}=I_k$ by pre- and post-multiplying both sides of (\ref{eq4})
by $\prod_{i=1}^{n}\det(\otimes^{r}T_{ii}^{-*})$ and $\prod_{i=1}^{n}\det(\otimes^{r}T_{ii}^{-1})$,
 respectively. Thus, it suffices to show that  
\begin{eqnarray}\label{e34} \det T_{n}^{r}(T^{*}T)\ge 1. \end{eqnarray}

We now prove (\ref{e34}) by induction. When $n=2$,
\begin{align*}
\det\big(T_2^{r}(T^*T)\big) 
&=\det \begin{bmatrix} \otimes^{r}I_{k} &  \otimes^{r}T_{12}  \\ 
\otimes^{r}T_{12}^{*}  &   \otimes^{r}(I_k+T_{12}^*T_{12}) \end{bmatrix} \\
&=\det \begin{bmatrix} I_{k^r} &  \otimes^{r}T_{12}  \\ 
\otimes^{r}T_{12}^{*}  &   \otimes^{r}(I_k+T_{12}^*T_{12})\end{bmatrix}\\
&=\det \begin{bmatrix} I_{k^r} &  \otimes^{r}T_{12}  \\ 
0  &   \otimes^{r}(I_k+T_{12}^*T_{12})-\otimes^{r}T_{12}^{*}\otimes^{r}T_{12}\end{bmatrix}\\
&=\det\left(\otimes^{r}(I_k+T_{12}^*T_{12})-\otimes^{r}(T_{12}^{*}T_{12})\right)\\
&\ge\det(\otimes^{r}I_k)= 1, 
\end{align*} 
in which the first inequality is by  Proposition \ref{prop21}. 

Suppose now (\ref{e34}) is true for $n=m$, and then consider the case $n=m+1$. 
For notational convenience, we denote $T= \begin{bmatrix} I_k &  V\\ 0&\widehat{T}\end{bmatrix}$, 
where $V= \begin{bmatrix}T_{12}&\cdots & T_{1n}\end{bmatrix}$ and $\widehat{T}=\Big[T_{i+1,j+1}\Big]_{i,j=1}^m$. 
Let $\widehat{V}=\begin{bmatrix}\otimes^{r}T_{12}&\cdots & \otimes^{r}T_{1n}\end{bmatrix}$. 
Clearly, by Proposition \ref{prop21}, $\widehat{V}^*\widehat{V}=T_{m}^{r}(V^{*}V)$.

Now computing
\begin{eqnarray*}
	T^*T=
	\begin{bmatrix} I_k &  V\\ 0&\widehat{T}\end{bmatrix}^*
	\begin{bmatrix} I_k &   V\\ 0&\widehat{T}\end{bmatrix}
	= \begin{bmatrix} I_k &  V\\ V^*&\widehat{T}^*\widehat{T}+V^*V\end{bmatrix}.
\end{eqnarray*}
Then 
\begin{eqnarray*}
	\det\big(T_n^{r}(T^*T)\big)&=&
	\det\begin{bmatrix}
		\otimes^{r}I_{k} &  \widehat{V}\\ \widehat{V}^*& T_m^{r}(\widehat{T}^*\widehat{T}+V^*V)
	\end{bmatrix}\\
	&=&\det\big(T_m^{r}(\widehat{T}^*\widehat{T}+V^*V)-\widehat{V}^*\widehat{V}\big)\\
	&=&\det\big(T_m^{r}(\widehat{T}^*\widehat{T}+V^*V)-T_{m}^{r}(V^{*}V)\big)\\	
    &\ge& \det\big(T_m^{r} (\widehat{T}^*\widehat{T}) +T_m^{r}(V^*V)-T_{m}^{r}(V^{*}V)\big)\\
	&=&\det\big(T_m^{r}(\widehat{T}^*\widehat{T})\big)\ge 1,
\end{eqnarray*}
in which the first inequality is by  Lemma \ref{lem33}, 
while the second one is by the induction hypothesis. 
Thus, (\ref{e34})  holds for $n=m+1$, 
so the proof of the induction step is complete.  
Hence we complete the proof of the proposition.   
\end{proof}

We now give the first extension of Fiedler-Markham's inequality (\ref{eqfm}) and (\ref{eqlin}). 

\begin{theorem}\label{thm35}
Let  $H=[H_{ij}]\in \mathbb{M}_{n}(\mathbb{M}_{k})$  be positive semidefinite.  
Then  for $r\in \mathbb{N}^*$ 
\begin{equation} \label{eq6}
\left(\frac{\det \bigl[ (\mathrm{tr} H_{ij})^r \bigr] }{k^{rn}}\right)^k \ge (\det { H})^r.
\end{equation}
\end{theorem}

\begin{proof}
The proof is a combination of   Theorem \ref{Lin} and Proposition \ref{prop34}.  
By Corollary \ref{coro32}, $T_{n}^{r}(H) \in 
\mathbb{M}_{n}(\mathbb{M}_{k^r})$ is positive semidefinite, 
then by  (\ref{eqlin}) of Theorem \ref{Lin}, we have
  \begin{eqnarray*}
 	\left(\frac{\det [ (\tr H_{ij})^r]}{k^{rn}}\right)^{k^r}
= \left(\frac{\det [\tr \otimes^{r}H_{ij}]}{k^{rn}}\right)^{k^r} 
\geq \det T_{n}^{r}(H),
 	\end{eqnarray*}
which together with Proposition \ref{prop34} leads to the following
 $$ \left(\frac{\det [ (\tr H_{ij})^r]}{k^{rn}}\right)^{k^r} 
\geq (\det H)^{ rk^{r-1}}.$$
Hence, the desired result (\ref{eq6}) follows.
\end{proof}

Obviously, when $r=1$, (\ref{eq6}) reduces to Fiedler and Markham's result (\ref{eqlin}). 
Using the same idea in the proof of Proposition \ref{prop34}, 
one could also get the following determinantal inequality for $Q_n^r(H)$.   
We omit the proof and leave the details for the interested reader.

 \begin{proposition}\label{prop36} 
Let $H=[H_{ij}]\in \mathbb{M}_{n}(\mathbb{M}_{k})$ be positive  definite. 
Then for $r\in \mathbb{N}^*$  
\begin{eqnarray*}\label{eq7}
 	 \det Q^{r}_n(H)\ge (\det H)^{ \frac{r}{k}{k+r-1\choose r}}.
 	\end{eqnarray*}  
\end{proposition}

We next show another extension of Fiedler-Markham's inequality similarly. 

 \begin{theorem}\label{thm37} 
Let  $H=[H_{ij}]\in \mathbb{M}_{n}(\mathbb{M}_{k})$  be positive semidefinite.  
Then for $r\in \mathbb{N}^*$  
	\begin{eqnarray}\label{eq8}
	\left(\frac{\det  [s_{r}(H_{ij})]}{ {{k+r-1}\choose {r}}^{n}}\right)^{k}\geq (\det H)^{r}.
	\end{eqnarray}  
\end{theorem}

\begin{proof}
 By  Corollary \ref{coro32} and Theorem \ref{Lin}, we obtain
  \begin{eqnarray*}
 	\left(\frac{\det [ s_r( H_{ij}) ]}{{{k+r-1}\choose {r}}^{n} }\right)^{{k+r-1 \choose r}}
= \left(\frac{\det [\tr \vee^{r}H_{ij}]}{ {{k+r-1}\choose {r}}^{n}}\right)^{{k+r-1 \choose r}} 
\geq \det Q_{n}^{r}(H),
 	\end{eqnarray*}
which together with Proposition \ref{prop36} yields  the following
 $$ \left(\frac{\det [ s_r(H_{ij}) ]}{{{k+r-1}\choose {r}}^{n} }\right)^{{k+r-1 \choose r}} 
\geq (\det H)^{ \frac{r}{k}{k+r-1\choose r}}.$$
Thus, the desired result (\ref{eq8}) follows.
\end{proof}

Clearly, when $r=1$, (\ref{eq8}) reduces to Fiedler and Markham's result (\ref{eqlin}).

\section{Extensions of Thompson's inequality}

\label{sec4}

Motivated by Theorem \ref{thm35} and Theorem \ref{thm37}, 
we apply Theorem \ref{thm11} to matrices $T_n^r(H)$ and $Q_n^r(H)$, respectively, 
and then combining with Proposition \ref{prop34} and Proposition \ref{prop36}, we have  
\begin{equation*}
\det \bigl[ \det \otimes^r H_{ij} \bigr] \ge \det T_n^r(H) \ge (\det H)^{rk^{r-1}},  
\end{equation*}
and 
\begin{equation*}
\det \bigl[ \det \vee^r H_{ij} \bigr] \ge \det Q_n^r(H) \ge (\det H)^{
\frac{r}{k}{k+r-1 \choose r}}.   
\end{equation*}
By Proposition \ref{prop21}, we get  the following extensions of Thompson's result (\ref{eq1}), 
\begin{equation} \label{eqth1}
\det  \bigl[ (\det  H_{ij})^{rk^{r-1}} \bigr]  \ge (\det H)^{rk^{r-1}},
\end{equation}
and 
\begin{equation} \label{eqth2}
\det  \Bigl[ (\det  H_{ij})^{\frac{r}{k}{k+r-1 \choose r}} \Bigr]  
\ge (\det H)^{\frac{r}{k}{k+r-1 \choose r}}. 
\end{equation}

At the end of this paper, we  present a more general setting of (\ref{eqth1}) 
and (\ref{eqth2}), i.e., we will relax the restriction of exponent,  
which also can be viewed as an extension of 
Thompson's inequality (\ref{eq1}) . 
Let $A $ and $B$ be complex matrix with the same size, 
we denote by $A\circ B$  the Hadamard product of $A,B$ 
and denote by $\circ^{r}A$ the $r$-fold Hadamard power of $A$.

\begin{theorem}\label{thm38} 
Let  $H=[H_{ij}]\in \mathbb{M}_{n}(\mathbb{M}_{k})$  be positive semidefinite. 
Then  for $r\in \mathbb{N}^*$ 
	\begin{eqnarray}\label{eq10}
	\det \bigl[ (\det H_{ij})^r \bigr] \geq (\det H)^r.
	\end{eqnarray}
\end{theorem} 

\begin{proof}
By Oppenheim's inequality \cite[p. 509]{HJ13}, we obtain 
\[ \det \bigl[ (\det H_{ij})^r \bigr] =\det \bigl( \circ^r [\det H_{ij}] \bigr)  
\ge \bigl(\det \,[\det H_{ij}]  \bigr)^r. \]
By (\ref{eq1}) of Theorem \ref{thm11}, we get 
\[ \bigl( \det \,[\det H_{ij}] \bigr)^r \ge (\det H)^r. \]
This completes the proof. 
\end{proof}

By taking the special case $n=2$ in (\ref{eq10}), we can easily get the following Corollary \ref{coro39}, 
which  is a generalization of Fischer's inequality  (\ref{eqfis}). 

\begin{corollary} \label{coro39}
Let  $H=[H_{ij}]\in \mathbb{M}_{2}(\mathbb{M}_{k})$ 
be positive semidefinite. Then  for $r\in \mathbb{N}^*$ 
\[ (\det H_{11} \det H_{22})^r- (\det H_{21} \det H_{12})^{r} \ge (\det H)^r.   \]
\end{corollary}

\section*{Acknowledgments}
The first  author would like to thank Dr. Minghua Lin  
for bringing the question to his attention. 
All authors are grateful for valuable comments from the referee, 
which considerably improve the presentation of our manuscript.
This work was supported by  NSFC (Grant Nos. 11671402, 11871479),  
Hunan Provincial Natural Science Foundation (Grant Nos. 2016JJ2138, 2018JJ2479) 
and  Mathematics and Interdisciplinary Sciences Project of CSU.


\begin{thebibliography}{20}
  \bibitem{Bha97} 
R. Bhatia, {\it Matrix Analysis}, 
GTM 169, Springer-Verlag, New York,	 1997.

\bibitem{CTZ19}
D. Choi, T. Y. Tam, P. Zhang, 
Extensions of Fischer's inequality, 
Linear Algebra Appl. 569 (2019) 311--322.

 \bibitem{Eve58}  
W. N. Everitt,  A note on positive definite matrices, 
Proc. Glasgow Math. Assoc. 3 (1958)   173--175.

\bibitem{SHGJ06}  
S. Fallat, A. Herman, M. Gekhtman and C.R. Johnson,  Compressions of totally positive matrices. 
SIAM J. Matrix Anal. Appl. 28 (2006)  68--80. 

\bibitem{FanKy}
K. Fan, 
Advanced Problems and Solutions: Solutions 4430, 
Amer. Math. Monthly 60 (1953) 50.

\bibitem{FM94} 
M. Fiedler, T. L. Markham,  On a theorem of Everitt, Thompson and de Pillis, 
Math. Slovaca  44 (1994) 441--444. 

\bibitem{HJ13} 
R. A. Horn, C. R. Johnson, {\it Matrix Analysis}, 
2nd ed., Cambridge University Press, Cambridge, 2013. 

\bibitem{JM85} 
C. R. Johnson, T. L. Markham, Compression and Hadamard power inequalities, 
Linear Multilinear Algebra 18 (1985) 23--34. 

\bibitem{Kua17} 
L. Kuai, An extension of the Fiedler-Markham determinant inequality, 
Linear Multilinear Algebra 66 (2018) 547--553. 

\bibitem{Li20}
Y. Li, L. Feng, Z. Huang, W. Liu, 
Inequalities regarding partial trace and partial determinant, 
Math. Ineq. Appl. 23 (2020) 477--485. 

\bibitem{LD14}
M. Lin, P. V. D. Driessche, 
Positive semidefinite $3\times 3$ block matrices, 
Electron. J. Linear Algebra 27 (2014) 827--836.

\bibitem{LS16} 
M. Lin, S. Sra,  A proof of Thompson's determinantal inequality, 
Math. Notes  99 (2016) 164--165. 

\bibitem{LinZhang}
M. Lin, P. Zhang, Unifying a result of Thompson, Fiedler and Markham, 
Linear Algebra Appl. 533 (2017) 380--385.

\bibitem{Lin16} 
M. Lin, A treatment of a determinant inequality of Fiedler and Markham, 
Czech. Math. J. 66 (2016) 737--742. 

\bibitem{Me97}
R. Merris, Multilinear Algebra, Gordon \& Breach, Amsterdam, 1997. 

\bibitem{deP71} 
J. de Pillis, Inequalities for partitioned positive semidefinite matrices, 
Linear Algebra Appl. 4 (1971) 79--94. 


\bibitem{Tho61} 
R. C. Thompson, A determinantal inequality for positive definite matrices, 
Canad. Math. Bull. 4 (1961) 57--62. 

 \bibitem{Zha12}
F. Zhang, Positivity of matrices with generalized matrix functions. 
Acta Math. Sin. (Engl. Ser.) 28 (2012)   1779--1786. 

\bibitem{Zhang11}
F. Zhang, 
Matrix Theory: Basic Results and Techniques, 2nd ed., 
Springer, New York, 2011. 

\end{thebibliography}
\end{document}